\documentclass[10pt]{amsart}
\usepackage{amscd}

\usepackage{amssymb}

\newcommand{\bP}{{\mathbb P}}

\newcommand{\bZ}{{\mathbb Z}}

\newcommand{\bC}{{\mathbb C}}

\newcommand{\bQ}{{\mathbb Q}}

\newcommand{\la}{{\langle}}
\newcommand{\ra}{{\rangle}}

\newtheorem{thm}{Theorem}[section]
\newtheorem{lemma}[thm]{Lemma}
\newtheorem{cor}[thm]{Corollary}
\newtheorem{prop}[thm]{Proposition}

\newtheorem{que}[thm]{Question}

\numberwithin{equation}{section}

\begin{document}

\title[]{Coniveau filtrations and  Milnor operation $Q_n$}
 
 \author{Nobuaki Yagita}

\address{[N. Yagita]
Department of Mathematics, Faculty of Education, Ibaraki University,
Mito, Ibaraki, Japan}

\email{nobuaki.yagita.math@vc.ibaraki.ac.jp}

\begin{abstract}
Let $BG$ be the classifying space of an algebraic group $G$
over the field $\bC$ of complex numbers.
We compute a new stable birational invariant defined by the difference of two coniveau filtrations  of a
 smooth projective approximation of $BG\times \bP^{\infty}$.
\end{abstract}

\maketitle

\section{Introduction}

Let $p$ be a prime number and $A=\bQ,\bZ$ or $\bZ/p^i$ for $i\ge 1$.
Let $X$ be a smooth algebraic variety over $k=\bC$.  
Let us recall the coniveau filtration
of the cohomology with coefficients in  $A$,
\[ N^cH^i(X;A)=\sum_{Z\subset X} ker(j^*:H^i(X;A)\to
H^i(X-Z,A))\]
where $Z\subset X$ runs through the closed subvarieties of codimension at least $c$ of $X$,
and $j:X-Z\subset X$ is the complementary open immersion.

Similarly, we can define the $strong$ coniveau filtration 
by 
\[ \tilde N^cH^i(X;A)=\sum_{f:Y\to X} im(f_*:H^{i-2r}(Y;A)\to
H^i(X,A))\]
where the sum is over all proper morphism $f: Y\to X$ from a smooth complex variety $Y$ of $dim(Y)=dim(X)-r$ with $r\ge c$,
and $f_*$ its transfer (Gysin map).
It is immediate that $\tilde N^cH^*(X;A)\subset 
N^cH^*(X;A)$.

It was hoped that when $X$ is proper, 
the strong coniveau filtration was  just the coniveau filtration,
i.e., 
$ \tilde N^cH^i(X;A)=N^cH^i(X;A)$.
In fact Deligne shows they are the same for $A=\bQ$.
However, Benoist and Ottem (\cite{Be-Ot}) recently show that they are not equal
for $A=\bZ$.

Let $G$ be an algebraic group such that $H^*(BG;\bZ)$ has
$p$-torsion for the classifying space $BG$
defined by Totaro \cite{To}.  
The purpose of this paper is to compute the $mod(p)$ stable birational
invariant of $X$ (Proposition 2.4 in \cite{Be-Ot})
\[ DH^*(X;A)=N^1H^*(X;A)/(p,\tilde N^1H^*(X;A)) \]
for  projective approximations $X$ for $BG$ (\cite{Ek},\cite{To},\cite{Pi-Ya}).
In fact, we see that $DH^*(X;\bZ)\not =0$ happen very frequently
in the above cases.

Here an approximation (for $degree \le N$) is 
the projective (smooth)  variety $X=X(N)$ such that
there is a map $g:X\to BG\times \bP^{\infty}$ with \[g^*:H^*(BG\times \bP^{\infty};A)\cong H^*(X;A) \quad 
for \ *<N. \]
 (In this paper, we say $X$ is an approximation $for$ $BG$ when it is that $of$ $BG\times \bP^{\infty} $
strictly speaking.)    Let us write $DH^*(X;\bZ)$ by $DH^*(X)$ simply as usual.

Here we give an example that we can compute a nonzero  $DH^*(X)$.  For  $G=(\bZ/p)^3$ the elementary abelian $p$-group of $rank=3$, we know  (for $p$ odd)
\[ H^*(BG;\bZ/p)\cong \bZ/p[y_1,y_2,y_3 ]\otimes \Lambda(x_1,
x_2,x_3),\quad |x_i|=1,\  Q_0(x_i)=y_i\]
\[H^*(BG)/p\cong \bZ/p[y_1,y_2,y_3](1,Q_0(x_ix_j),Q_0(x_1x_2x_3)|1\le i<j\le 3) \]
where $Q_0=\beta$ is the Bockstein operation, and the notation $R(a,...,b)$ (resp. $R\{a,..., b\}$) means the $R$-submodule (resp. the free $R$-module) generated by $a,...,b$.
\begin{thm}  Let $G=(\bZ/p)^3$.  Given $N>2p+3$, there are  approximations
$X=X(N)$ for  $BG$ such that  we have for  $\alpha_{ij}=Q_0(x_ix_j),$   $\alpha=Q_0(x_1x_2x_3)$
\[ DH^*(X)\cong DH^3(X)\oplus DH^4(X)\]
\[\cong \bZ/p\{\alpha_{ij},\alpha|1\le i<j\le 3\}
\quad when\  *<N
.\]
But we have  $DH^*(X;\bZ/p)=0$ when $*<N$.
\end{thm}

Benoist and Ottem  give  examples,  where the above 
invariants are  nonzero  for $A=\bZ_{(2)}$ by using the Steenrod squares.  
We note that some of their arguments can be extended for $A=\bZ_{(p)}$ by using the
Milnor operation $Q_n$ for odd primes $p$. However it seems not so easy to give a nontrivial example 
for $A=\bZ/p$ in the case $X$ is an approximation for $BG$
as the above examples show.

For connected groups we have
\begin{thm}  Let $G$ be a simply connected group such that $H^*(BG)$ has $p$-torsion.
Then there is an approximation $X$ for $BG$ such that 
   $ DH^4(X)\not =0.$
\end{thm}
	\begin{thm}
Let $p$ be an odd prime number, and
$G=PGL_p$. 
Then there is an  approximation $X$ for $BG$ such that  
$DH^3(X)\not =0$. \end{thm}
\begin{thm} Let $p=2$ and $X_n=X_n(N)$ be approximations for $BSO_n$  with $n\ge 3$.  Then for sufficient large $N$
\[  DH^*(X_{2m+1})\supset  \bZ/2\{w_3,w_5,...,w_{2m+1}\}
\quad for \ *<N\]
where $w_i$ is the $i$-th Stiefel-Whiteny class.
\end{thm}

\section{transfer and $Q_n$}

The Milnor operation (in $H^*(-;\bZ/p)$)
is defined by $Q_0=\beta$ and for $n\ge 1$
\[ Q_n=P^{\Delta_n} \beta-\beta P^{\Delta_n}, \quad \Delta_n=(0,..,0,\stackrel{n}{1},0,...) \]
(For details see \cite{Mi}, \S 3.1 in \cite{Vo1})
where $\beta$ is the Bockstein operation and $P^{\alpha}$
for $\alpha=(\alpha_1,\alpha_2,...)$ is the fundamental base of the module of finite sums of products of reduced powers.
. 
\begin{lemma} Let $f_*$ be the transfer (Gysin) map (for 
 proper smooth) $f:X\to Y$.
Then  $Q_nf_*(x)=f_*Q_n(x)$ for $x\in H^*(X;\bZ/p)$.
\end{lemma}

The above lemma is known (see the proof of Lemma 7.1 in \cite{Ya2}). The transfer $f_*$ is 
expressed as $g^*f_*'$ such that
\[ f_*'(x)=i^*(Th(1)\cdot x), \quad x\in H^*(X;\bZ/p)\]
for some maps $g,f', i$ and the Thom class $Th(1)$.  Since 
$Q_n(Th(1))=0$ and $Q_i$ is a derivation, we get the lemma.
However, we give here the another computational proof.

\begin{proof}[Proof of Lemma 2.1.]
Recall the following  Grothendieck formula (e.g., [Q1])
\[(1)\quad P_t(f_*(x))=f_*(c_t\cdot P_t(x)).\]
Here the total reduced powers $P_t(x)$ are defined 
\[P_t(x)=\sum_{\alpha}P^{\alpha}(x)t^{\alpha}\in H^*(X;\bZ/p)[t_1,t_2,...]\quad with 
\ t^{\alpha}=t_1^{\alpha_1}t_2^{\alpha_2}...\]
where $\alpha=(\alpha_1,\alpha_2,...)$ and $degree(t^{\alpha})=
\sum_i 2\alpha_i(p^i-1)$ (each element in $H^*(X;\bZ/p)$ is represented as a homogeneous part respective  to the above degree).
The total Chern class $c_t$ is defined similarly, for the Chern classes of the normal bundle of the map $f$.

We consider the above equation with the assumption such that 
$t_n^2=0$ and $t_j=0$ for $j\not =n$, i.e., 
$P_t(x)\in H^*(X;\bZ/p)\otimes \Lambda(t_n)$. 
That means
\[(2)\quad  P_t(f_*(x))=(1+P^{\Delta_n}t_n)(f_*(x))\]
\[ (3)\quad f_*(c_t\cdot P_t(x))=f_*((1+c_{p^n-1}t_n)(x+P^{\Delta_n}(x)t_n))
=f_*(x+(c_{p^n-1}x+P^{\Delta_n}(x))t_n).\]
From (1), we see $(2)=(3)$ and we have
\[ (4)\quad P^{\Delta_n}(f_*(x))=f_*(c_{p^n-1}x+P^{\Delta_n}(x)).\]

By the definition, $\beta$ commutes with $f_*$, and we have
\[(5)\quad  P^{\Delta_n}\beta(f_*(x))=P^{\Delta_n}f_*(\beta x)=
f_*(c_{p^n-1}\beta x+P^{\Delta_n}(\beta x)).\]

On the other hand 
\[(6)\quad \beta P^{\Delta_n}f_*(x)\stackrel{(4)}{=}\beta f_*(c_{p^n-1}x+P^{\Delta_n}(x))
=f_*(c_{p^n-1}\beta x+\beta P^{\Delta_n}(x)).\]
Then $(5)-(6)$ gives  that
$(P^{\Delta_n} \beta-\beta P^{\Delta_n})f_*(x)=f_*(P^{\Delta_n} \beta-\beta P^{\Delta_n})(x).$
\end{proof}

By the definition, each cohomology operation $h$ (i.e., an element in the Steenrod algebra) is written (with $Q^B=Q_0^{b_0}Q_1^{b_1}...$)
by
\[h=\sum_{A,B}P^AQ^B\quad with\ A=(a_1,...),\ B=(b_0,...)\
b_i=0\ or\  1.\]

\begin{cor}  We have \ 
$ P_tQ^B(f_*(x))=f_*(c_t\cdot P_tQ^B(x)).$
\end{cor}
Hence cohomology operations $h$ (for $H^*(-;\bZ/p))$  which commute with
all transfer $f_*$ are cases $c_t=1$, i.e. $A=0$ which are only products $Q^B$ of Milnor operations $Q_i$.

\section{coniveau filtrations}

Bloch-Ogus \cite{Bl-Og} give a spectral sequence such that its $E_2$-term is given by
\[E(c)_2^{c,*-c}\cong 
H_{Zar}^c(X,\mathcal{H}_{A}^{*-c})\Longrightarrow H_{et}^*(X;A)\]
where $\mathcal{H}_{A}^*$ is the  Zariski sheaf induced from the  presheaf  given by $U\mapsto H_{et}^*(U;A)$
for an open $U\subset X$.

The filtration for this spectral sequence is defined as the coniveau filtration 
  \[N^cH_{et}^*(X;A)= F(c)^{c,*-c}\]
where the infinite term $ E(c)^{c,*-c}_{\infty}\cong  F(c)^{c,*-c}/F(c)^{c+1,*-c-1}$ and 
\[ N^cH^*_{et}(X;A)=\sum_{Z\subset X\ ;\ codim_X(Z)\le c} ker(j^*:H^*_{et}(X;A)\to
H^*_{et}(X-Z,A)).\]

Here we recall the motivic cohomology 
$H^{*,*'}(X;\bZ/p)$ defined by Voevodsky and Suslin (\cite{Vo1},\cite{Vo3},\cite{Vo4})
so that
\[ H^{i,i}(X;\bZ/p)\cong H^i_{et}(X;\bZ/p)\cong H^i(X;\bZ/p).\]

Let us write  
$H^*_{et}(X;\bZ )$ simply 
by $H^*_{et}(X)$ as usual.
Note that we do $not$ assume  $H^*_{et}(X) \cong H^*(X)$
while we have the natural map $H_{et}^*(X)\to H^*(X)$.

Let $0\not =\tau \in H^{0,1}(Spec(\bC);\bZ/p)$.  Then by the  multiplying by $\tau$, we can define a map
$H^{*,*'}(X;\bZ/p)\to H^{*,*'+1}(X;\bZ/p)$.
By Deligne ( foot note (1) in Remark 6.4 in \cite{Bl-Og}) and
Paranjape (Corollary 4.4 in \cite{Pa}), it is proven that
there is an isomorphism of the coniveau spectral sequence
with the $\tau$-Bockstein spectral sequence $E(\tau)_r^{*,*'}$
(see also \cite{Te-Ya}, \cite{Ya1}).

\begin{thm} (Deligne)
Let $A=\bZ/p$.  Then we have 
the isomorphism of spectral sequence 
\[E(c)_r^{c,*-c}\cong E(\tau)_{r-1}^{*,*-c} \quad 
for \ r\ge 2.\] 
Hence the filtrations are the same, i.e. 
$N^cH_{et}^*(X;\bZ/p)= F_{\tau}^{*,*-c}$
where 
\[F_{\tau}^{*,*-c}=Im(\times \tau^c:H^{*,*-c}(X;\bZ/p)\to H^{*,*}(X;\bZ/p)).\]
\end{thm}

\begin{lemma}
Let $x\in H^{*,*}(X)$ and also in $N^cH^*(X;\bZ/p)$. 
Then if  the map $H^{*+1,*-c}(X)\to H^{*+1,*}(X)$
is injective,  then $x\in N^cH^*(X)$.
\end{lemma}
\begin{proof} Consider the exact sequences
\[ \begin{CD} @>{p}>>{ x''\in H^{*,*-c}(X)} @>{r_1}>> 
x'\in H^{*,*-c}(X;\bZ/p)
@>{\delta_1}>>  H^{*+1,*-c}(X) \\
  @.  @VV{f_1}V  @VV{f_2}V @VV{f_3}V  \\
@>{p}>> x\in H^{*,*}(X) @>{r_2}>> H^{*,*}(X;\bZ/p)
@>{\delta_2}>> H^{*+1,*}(X)
\end{CD} \]
By the assumption of this lemma,  we can take $x'\in H^{*,*-c}(X;\bZ/p)$
such that $r_2(x)=f_2(x')$. So $\delta_2f_2(x')=0$.  Since $f_3$ is injective,
we see $\delta_1(x')=0$,  Hence there is $x''\in H^{*.*-c'}(X)$
such that $r_1(x'')=x'$.
\end{proof}
Let $cl: CH^*(X)\otimes A\to H^{2*}(X;A)$ be the cycle map.
and $Im(cl)^+$ is the positive degree parts of its image.

\begin{lemma} We see that $Im(cl)^+\subset N^*H^{2*}(X;A)$.
\end{lemma}
\begin{proof}  Recall that $H^{*.*'}(X;A)\subset N^{*-*'}H^*(X;A)$.  We have $H^{2*,*}(X;A)\cong CH^*(X)\otimes A.$
Since $2*>*$ for $*\ge 1$, we see $cl(y)\in N^{1}H^{2*}(X;A)$.
\end{proof}

Hence $y\in CH^*(X)\otimes A$ is represented by closed algebraic set 
supported $Y$, while $Y$ may be singular.

On the other hand,  by Totaro \cite{To}, we have the modified cycle map $\bar cl$
\[ cl:CH^*(X)\otimes A\stackrel{\bar cl}{\to} MU^{2*}(X)
\otimes_{MU^*}A\stackrel{\rho}{\to} H^{2*}(X;A)\]
for the complex cobordism theory $MU^*(X)$.
It is known \cite{Qu1} that elements in $MU^{2*}(X)$ is written by 
a proper map to X from a  stable almost complex manifold $Y$.
while it does not mean complex manifold.

The following lemma is well known.
\begin{lemma}
If $x\in Im(\rho)$ for $\rho: MU^*(X)/p\to H^*(X;\bZ/p)$, then 
we have $Q_i(x)=0$ for all $i\ge 0$. 
\end{lemma}
\begin{proof}
In fact, it is known that  $d=Q_i$ for the first nonzero differential $d$
of the Atiyah-Hirzebruch spectral sequence converging to the connected 
Morava K-theory $k(i)^*(X)$ with $k(i)^*=\bZ/p[v_i]$. We have the natural maps
\[ \rho:MU^*(X)/p \stackrel{\rho_1}{\to} 
 k(i)^*(X)\stackrel{\rho_2}{\to} H^*(X:\bZ/p).\]
 Hence  $Q_i\rho_2(x)=0$
implies $Q_i\rho(x)=0$.
 \end{proof}

\begin{lemma} (reciprocity law)
If $a\in \tilde N^*H^{2*}(X;A)$, then for each $g\in DH^{*'}(X;A)$ we have
$ag\in \tilde N^*H^{2*+*'}(X;A)$.
\end{lemma}
\begin{proof}
Suppose we have $f:Y\to X$ with $f_*(a')=a$.
Then \[  f_*(a'f^*(g))=f_*(a')g=ag\]
 by Frobenius reciprocity law.
\end{proof} 

Let $G$ be an algebraic group (over $\bC$) and $r$
be  a complex representation $r:G\to U_n$ the unitary group.  Then  we can define
the Chern class $c_i=r^*c_i^U$. Here the Chern classes
$c_i^U$  
in  $H^*(BU_n)\cong \bZ[c_1^U,...,c_n^U]$ are defined by
using  the Gysin map as $c_n^U=i_{n,*}(1)$ for 
\[ i_n^* : H^*(BU_n)\cong H_{U_n}^*(pt.)\stackrel{i_{n,*}}{\to} H_{U_n}^{*+2i}(\bC^{\times i})\cong H^{*+2i}(BU_n)\]
where $H_{U_n}(X)=H^*(EU_n\times_{U_n}X)$ is the $U_n$-equivaliant cohomology.

\begin{lemma}
For an approximation $X$ of $BG$ and 
for each $g\in H^{*'}(BG;A)$, we have
$gc_i=0\in DH^*(X;A)$.
\end{lemma}

The following lemma is proved by 
 Colliot Th\'er\`ene and Voisin
\cite{Co-Vo} by using the
affirmative answer of the Bloch-Kato conjecture
by Voevodsky. (\cite{Vo3}. \cite{Vo4})
\begin{lemma}  (\cite{Co-Vo})  Let $X$ be a smooth complex variety.
Then any torsion element in $H^*(X)$ is in
$N^1H^*(X)$.
\end{lemma}

\section{the main lemmas}

The following lemma is the $Q_i$-version of  one of results by Benoist and Ottem.
\begin{lemma}
Let $\alpha\in N^1H^s(X)$ for $s=3$ or $4$.
If $Q_i(\alpha)\not =0\in H^*(X;\bZ/p)$ for some $i\ge 1$,  then
\[ DH^s(X)\supset \bZ/p\{\alpha\},\quad 
DH^s(X;\bZ/p^t)\supset \bZ/p\{\alpha\}\ \ for\ t\ge 2.
\]
\end{lemma}
\begin{proof}
Suppose $\alpha\in \tilde N^1H^s(X)$ for $s=3$ or $4$, i.e.
there is  a smooth $Y$ with $f:Y\to X$ such that
the transfer $f_*(\alpha')=\alpha$ for $\alpha'\in H^*(Y)$.
Then for $s=4$
\[ Q_i(\alpha')=(P^{\Delta_i}\beta-\beta P^{\Delta_i})(\alpha')=(-\beta P^{\Delta_i})(\alpha') 
 = -\beta (\alpha')^{p^i}\]
\[ =-p^i(\beta \alpha')(\alpha')^{p^i-1}=0 \quad (by\ Cartan formula)\]
since $\beta(\alpha')=0$ and $P^{\Delta_i}(y)=y^{p^i}$ for $dim(y)=2$.
(For $s=3$, we get also $Q_i(\alpha')=0$ since  $ P^{\Delta_i}(x)=0$ for $dim(x)=1$.)
This contradicts to the commutativity of $Q_i$ and $f_*$.
 
The case $A=\bZ/p^t$, $t\ge 2$ is proved similarly,  since 
for $\alpha'\in H^*(X;A)$ we see $\beta \alpha'=0\in H^*(X;\bZ/p)$.
\end{proof}

We  will extend Lemma 4.1, $s=4$, for the larger $s$, by using 
$MU$-theory of Eilenberg-MacLane spaces.

Recall that $K=K(\bZ,n)$ is the Eilenberg-MacLane space
such that the homotopy group $[X,K]\cong H^n(X:\bZ)$, i.e., each element 
$x\in H^n(X;\bZ)$ is represented by a homotopy map 
$x; X\to K$. Let $\eta_n\in H^n(K;\bZ)$ 
corrsponding the identity map.
(For $K'=K(\bZ/p,n)$ define $\eta_n'\in H^n(K';\bZ/p)$ by the identity  element of $K'$.)
We know the image $\rho(MU^*(K))\subset H^*(K;\bZ)/p$.

\begin{lemma} (\cite{Ta}, \cite{Ra-Wi-Ya})
We have the isomorphism 
\[  \rho :  MU^*(K)\otimes _{MU^*}\bZ/{p}
\cong\bZ/{p}[Q_{i_1}...Q_{i_{n-2}}\eta_n
|0<i_1<...<i_{n-2}],\]
\[  \rho :   MU^*(K')\otimes _{MU^*}\bZ/{p}
\cong\bZ/{p}[Q_{i_1}...Q_{i_{n-1}}Q_0\eta_{n}'
|0<i_1<...<i_{n-1}] , \]
where the notation $\bZ/p[a,...]$ exactly means $\bZ/p[a,...]/(a^2|\ |a|=odd)$. 
\end{lemma}

The following lemma is an extension of Lemma 4.1 for $s\ge 4$.
\begin{lemma}
Let $\alpha\in N^cH^{n+2c}(X)$, $n\ge 2$, $c\ge 1$. 
Suppose that  there is a sequence $0<i_1<...<i_{n-1}$ with 
\[ Q_{i_1}...Q_{i_{n-1}}\alpha \not =0 \quad   in\ H^*(X;\bZ/p).\]
Then $D^cH^{*}(X)=N^cH^{*}(X)/(p,\tilde N^cH^{*}(X)) \supset \bZ/p\{\alpha\}$.
\end{lemma}
\begin{proof}
Suppose $\alpha\in \tilde N^cH^{n+2c}(X)$,  i.e.
there is  a smooth $Y$ of $dim(Y)=dim(X)-c$ with $f:Y\to X$ 
such that 
the transfer $f_*(\alpha')=\alpha$ for $\alpha'\in H^n(Y)$.

Identify  the map $\alpha': Y\to K$ with $\alpha'=(\alpha')^*\eta_n.$
We still see from Lemma 4.2,
   \[Q(\alpha')= Q_{i_1}...Q_{i_{n-2}}((\alpha')^*\eta_n)\in 
Im(MU^*(Y)\to H^*(Y;\bZ/p)).\]
From  Lemma 3.4,  we see
\[Q_{i_{n-1}}Q(\alpha')=Q_{i_{n-1}}Q_{i_1}...Q_{i_{n-2}}(\alpha')=0 \in H^*(Y;\bZ/p).\]  

Therefore $Q_{i_{n-1}}Q(\alpha)$
must be zero by the commutativity of $f_*$ and $Q_i$.
\end{proof}

   {\bf Remark.}  For $\alpha\in N^cH^{n+2c}(X;\bZ/p)$, we can
give the $A=\bZ/p$ version  of the above lemma from the second isomorphism in Lemma 4.2. But  we can see
$Q_{i_1}...Q_{i_n}Q_0\alpha=0$ always, (even if $Q_{i_1}...Q_{i_{n-1}}\alpha\not=0$).  Hence we do not say about 
$D^cH^*(X;\bZ/p)$. 

\section{classifying spaces for finite groups }

Let $G$ be a finite group or an algebraic group, and $BG$ its classifying space.
For example, when $G=G_m$ is the multiple group, we see
\[ BG_m=BS^1\cong \bP^{\infty}, \quad 
H^*(\bP^{\infty})\cong \bZ[y]\ \  with\  |y=c_1|=2,\]
for the infinite (complex) projective space $\bP^{\infty}$.  Note that $BG_m$ is a colimit of complex projective spaces.
Though $BG$ itself is not a colimit of complex projective varieties, we can take a  complex projective variety
$X(N)$ (\cite{Ek}) for a given $N\ge 3$ such that there is a map $ j:X(N)\to BG\times \bP^{\infty} $  with
\[ H^*(BG\times \bP^{\infty};A)\stackrel{j^*}{\cong} H^*(X(N);A)\quad for \ 
*<N.\]

Note that the quotient
\[ N^nH^*(X;A)/(\tilde N^nH^*(X;A)) \]
 is invariant under replacing 
$X$ with $X\times \bP^m$ for all $n$ and all abelian groups $A$.  In fact, from K$\ddot{u}$nneth formula,
\[ H^*(X\times \bP^m; A)\cong H^*(X;A)\otimes \bZ[y]/(y^{m+1}),\]
where $y\in CH^1(\bP^m)$. Let  $Ideal(y)$ be the ideal of
$H^*(X\times \bP^m;A)$ generated by $y$. Then  $Ideal(y)\subset  \tilde N^*H^*(X\times \bP^{m};A)$ by the Frobenius reciprocity law  (Lemma 3.5). 
Moreover Benoist and Ottem  show the above quotient
when $n=1$ is a stable birational invariant (Proposition 2.4 in \cite{Be-Ot}).

In this paper, we say the above $X(N)$ as a ($degree\le N$) complex  projective approximation $for$ $BG$ (which  is an approximation $of$ $BG\times \bP^{\infty}$ strictly speaking), and
let us write the ($mod(p)$) stable rational invariant  by 
\[DH^*(X;A)=N^1H^*(X;A)/(p,\tilde N^1H^*(X;A)).\]

Hereafter, we consider $DH^*(X;A)$ when $A=\bZ$.
Let $p$ be an odd prime. ( The case $p=2$ is different but  a similar argument works.)
Let $G=(\bZ/p)^3$ the $rank=3$ elementary abelian $p$-group.
Then the $mod(p)$ cohomology is 
 \[H^*(BG;\bZ/p)\cong H^*(B\bZ/p;\bZ/p)^{3\otimes }\cong \bZ/p[y_1,y_2,y_3]\otimes \Lambda(x_1,x_2,x_3).\]
Here $|y_i|=2, |x_i|=1, \beta(x_i)=y_i$, and $\Lambda(a,...,b)$ is the $\bZ/p$-exterior algebra generated by $a,...,b$.

The integral cohomology 
(modulo $p$) 
is isomorphic to
\[ H^*(BG)/p\cong Ker(Q_0)\cong H(H^*(BG;\bZ/p);Q_0)\oplus
Im(Q_0)\]
where $H(-;Q_0)=Ker(Q_0)/Im(Q_0)$ is the homology with the differential $Q_0$.
It is immediate that  $H(H^*(B\bZ/p;\bZ/p); Q_0)\cong \bZ/p$.
By the K$\ddot{u}$nneth formula, we have 
$H(H^*((BG;\bZ/p);Q_0)\cong ( \bZ/p)^{3\otimes}\cong \bZ/p$.
Hence we have 
\[ H^*(BG)/p\cong \bZ/p\{1\}\oplus Im(Q_0)\]\[ \cong \bZ/p[y_1,y_2,y_3](1,Q_0(x_ix_j),
Q_0(x_1x_2x_3)|i<j)\]
where the notation $R(a,...,b)$ (resp. $R\{a,..., b\}$)
means the $R$-submodule (resp. the free $R$-module)
generated by $a,...,b$. Here we note $H^+(BG)$ is just $p$-torsion.

Also note that $y_1,y_2,y_3$ are represented by the Chern classes $c_1$.  From Lemma 3.6, we see
\[  Ideal(y_1,y_2,y_3)=0\in DH^*(X).\]

We know $Q_i(y_j)=y_j^{p^i}$ and $Q_j$ is a derivation.  Let us write 
\[ \alpha=Q_0(x_1x_2x_3)=y_1x_2x_3-y_2x_1x_3+y_3x_1x_2.\]
Note $\alpha\in H^4(X)$, $p\alpha=0$,  and $\alpha\in 
N^1H^*(X)$ from Lemma 3.7.  Moreover
 \[Q_1(\alpha)=Q_1(y_1x_2x_3)-...=y_1y_2^px_3-y_1y_3^px_2-...\not =0 \in H^*(X;\bZ/p).\]
Similarly, for $\alpha_{ij}=Q_0(x_ix_{j})$, we see $Q_1(\alpha_{ij})\not =0$.
Hence from Lemma 4.1 and Lemma 3.6, we have 
\begin{thm}  Let $X=X(N)$ with  $N> 2p+3$ be an approximation for $B(\bZ/p)^3$.   Then we have 
\[DH^*(X)\cong 
\bZ/p\{\alpha_{ij}, \alpha|1\le i<j\le 3\}\quad for\ *<N.\]
\end{thm}
\begin{proof}
We see $H^*(BG)/(p,y_1,y_2,y_3)\cong \bZ/p\{1,\alpha_{ij},\alpha\}$. Of course $1\not \in N^1H^*(X)$, we have the theorem from Lemma 4.1.
\end{proof} 
\begin{thm}  Let $X=X(N)_n$ be an approximation for $(B\bZ/p)^n$ with $N>|Q_0Q_1...Q_{n-1}(x_1...x_n)|$.
 Then
we have  for $\alpha_{i_1,...,i_s}=Q_0(x_{i_1}...x_{i_s})$,
\[DH^*(X)\supset 
\bZ/p\{(\alpha_{i_1,...,i_s }|2\le s,\  0<i_1<i_2...<i_s\le n\}\quad for\ *<N.
\]
\end{thm}
\begin{proof}  We have the theorem from Lemma 4.3 and
$ Q_{i_1}...Q_{i_{s-2}}(\alpha_{i_1,...,i_s})$ is \[ Q_{i_1}...Q_{i_{s-2}}Q_0(x_{i_1}...x_{i_s})
=y_{i_1}^{p^{i_1}}...y_{i_{s-2}}^{p^{i_{s-2}}}
y_{i_{s-1}}x_{i_s}+...\not =0.\]
(Note the $n=|\alpha'|$ in Lemma 4.3 is written by $s-1$ here.)
\end{proof}

\begin{cor}  If $n\not =m\ge 3$, then $X(N)_n$ and $X(N)_m$
are not stable birational equivalent.
\end{cor}


Next we study small non abelian $p$-groups.
Let $G$ be a non abelian group of order $p^3$ (see $\S 8$,
for details).  Then  
$H^{even}(BG)$ is generated by Chern classes, and
$H^{odd}(BG)$ is a (just) $p$-torsion.  We can identify
$H^{odd}(BG)\subset H^{odd}(BG;\bZ/p)$.The operation $Q_1$
acts on $H^{odd}(X)$, and  induces the injection
\[    Q_1\ :\  H^{odd}(BG)\hookrightarrow    H^{even}(BG).\]

Such groups are four types (see $\S 8$ below), and they are called extraspecial $p$-groups $G=p_{\pm}^{1+2}$ of order $p^3$.
When $G=Q_8=2_-^{1+2}$ the quaternion group of order $8$, we know  $H^{odd}(X)=0$.
However when $G=D_8=2_+^{1+2}$ the dihedral group of order $8$, the cohomology $H^{odd}(BG)$ is generated as an $H^{even}(BG)$ module by an element $e$ of $deg(e)=3$.
When $G=E=p_+^{1+2}$ for $p\ge 3$, $H^{odd}(BG)$ is generated by
$e_1,e_2$ with  $deg(e_i)=3$.
When $G=M=p_-^{1+2}$ for $p\ge 3$, $H^{odd}(BG)$ is generated by $e'$ but  $deg(e')=2p+1$.

From Lemma 3.5 (Frobenius reciprocity) and the main lemma
(Lemma 4.1),
we have the following theorem. 

\begin{thm}  Let $X=X(N)$ with  $N>2p+3$ be an approximation for an extraspecial $p$-group $G$ of 
order $p^3$.  Then
we have for $*<N$
\[ DH^*(X)\cong 
\begin{cases} 0\quad for \ G=Q_8\\
           \bZ/2\{e\}\quad for \ G=D_8\\
               0\ or \ \bZ/p\{e'\} \quad for\ G=M \\
\bZ/p\{e_1,e_2\} \quad for \ G=E.
\end{cases} \]
\end{thm}

In particular, the above theorem implies that when $G=p_+^{1+2}
$, there is an $X$ with $DH^3(X)\not =0$ but for all $X=X(N)$ we see 
$ DH^*(X)=0$ for  $4\le *<N.$

\section{connected groups}

At first, we consider when $G=U_n$, $SU_n$  or $Sp_{2n}$
for all $p$,  where 
the cohomology $H^*(BG)$ has no torsion.
Then $H^*(BG)$ is generated by Chern classes, e.g., 
\[H^*(BU_n)\cong CH^*(BU_n)\cong \bZ_{(p)}[c_1,...,c_n],
\]
\[ H^*(BSp_{2n})\cong CH^*(BSp_{2n})\cong 
\bZ_{(p)}[c_2,c_4,...,c_{2n}].\]
 Hence $DH^*(X)=0$
for the approximations $X$ for these groups.

Next we consider the case $G=SO_3$ and  $p=2$.  Then
\[ H^*(BG;\bZ/2)\cong \bZ/2[w_1,w_2,w_3]/(w_1)\cong \bZ/2[w_2,w_3],\] 
where $w_i$ is the $i$-th Stiefel-Whitney class for $SO_3\subset O_3$ and $w_i^2=c_i$
is the $i$-th Chern class for $SO_3\subset U_3$.
(Also it is the elementary symmetric polynomial in $\bZ/2[y_1,...,y_i]$.)
 Here we know 
$Q_0(w_2)=w_3,$ and $Q_1(w_3)=w_3^2=c_3$.  Therefore we have 
\cite{Ya1}
\[ H^*(BG;\bZ/2)\cong \bZ/2[c_2,c_3]\{1,w_2,w_3=Q_0(w_2),w_2w_3=Q_1w_2\}\]
\[\cong \bZ/2[c_2,c_3]\otimes \Lambda(Q_0,Q_1)\{w_2\}
\oplus \bZ/2[c_2].\]
In particular $H^*(BG)/2\cong Ker(Q_0)\cong \bZ/2[c_2,c_3]\{1,w_3\}. $
Then from Lemma 4.1, we have 
\begin{thm} Let $G=SO_3$ and $X$ be an approximation of $BG$
for $6<N$.  Then 
$DH^*(X)\cong\bZ/2\{w_3\}$ for $*<N.$
\end{thm}

Using Lemma 4.3, we have

\begin{thm} Let $X_n=X_n(N)$ be approximations for $BSO_n$  for $n\ge 3$.  Moreover let $|Q_1...Q_{2m-1}(w_{2m+1})|<N$.Then
\[  DH^*(X_{2m+1})\supset  \bZ/2\{w_3,w_5,...,w_{2m+1}\}
\quad for \ *<N.\]
\end{thm}
\begin{proof}
Since $Q_0w_{2i}=w_{2i+1}$, we see $w_{2i+1}\in N^1H^{2i+1}(X)$.
We have the theorem, 
 from Lemma 4.3 and the restriction to $H^*(B(\bZ/2)^{2i};\bZ/2),$
\[ Q_1...Q_{2i-2}(w_{2i+1})=Q_1...Q_{2i-2}Q_0(w_{2i})
=y_1y_2^2...y_{2i-1}^{2^{2i-2}}x_{2i}+...\not =0.\]
\end{proof}

We next consider simply connected groups. Let us write by $X$ an 
approximation of $BG_2$ for the exceptional simple 
group $G_2$ of $rank=2$.
The $mod(2)$ cohomology is generated by the Stiefel-Whitney classes $w_i$ of
the real representation $G_2\to SO_7$
\[ H^*(BG_2;\bZ/2)\cong \bZ/2[w_4,w_6,w_7],
\quad P^1(w_4)=w_6,\ Q_0(w_6)=w_7, \]
\[H^*(BG_2)\cong (D'\oplus D'/2[w_7]^+)\quad where \ \ D'=\bZ[w_4,c_6].\]
Then we have  $Q_1w_4=w_7, Q_2(w_7)=w_7^2=c_7$ (the Chern class).

The Chow ring of $BG_2$ is also known
\[ CH^*(BG_2)\cong (D\{1,2w_4\} \oplus D/2[c_7]^+)\quad where \ \ D=\bZ[c_4,c_6]\ \ c_i=w_i^2.\]
In particular the cycle map $cl: CH^*(BG)\to H^*(BG)$ is injective.

It is known $w_4\in N^1H^*(X;\bZ/2)$ (\cite{Ya1})
and from Lemma 3.2, we see $w_4\in N^1H^*(X)$.
Since $Q_1(w_4)=w_7\not =0$,  from Lemma 4.1,
we have
$DH^4(X)\not =0$ (\cite{Be-Ot}).  Moreover
 $H^*(BG)/(c_4,c_6,c_7)\cong 
\Lambda(w_4,w_7)$ implies 
\begin{prop} For  $X$ an approximation for $BG_2$,
we have the surjection
\[ \Lambda(w_4,w_7)^+
\twoheadrightarrow 
DH^*(X)\quad for\ *<N.\]
\end{prop}

By Voevodsky \cite{Vo1}, \cite{Vo2}, we have  the $Q_i$ operation also in the 
motivic cohomology $H^{*,*'}(X;\bZ/p)$ with
$deg(Q_i)=(2p^i-1,p-1)$.  Then we can take 
\[ deg(w_4)=(4,3),\ deg(w_6)=(6,4),\ deg(w_7)=(7,4),\ deg(c_7)=(14,7).\]
By Theorem 3.1, the above means
\[ w_7=Q_1w_4\in N^{7-4}H^*(X;\bZ/2)=N^{3}H^*(X;\bZ/2).\]
We can not see here that $0\not =w_7\in DH^*(X)$,
but see the following proposition.

\begin{prop}  For $X$ the approximation for  $BG_2$, we have
\[\bZ/2\{w_7\}\subset D^3H^*(X)=N^3H^*(X)/(2,\tilde N^3H^*(X)).\]
\end{prop}
\begin{proof}
Suppose $w_7\in \tilde N^3H^*(X)$.  That is, there is
$x\in H^1(Y)$ with $f_*(x)=w_7$ for $f:Y\to X$.
Act $Q_2$ on $H^*(Y;\bZ/2)$, and
\[ Q_2(x)=(P^{\Delta_2}\beta +\beta P^{\Delta_2})(x)=0\]
since $\beta(x)=0$ and $P^{i}(x)=Sq^{2i}(x)=0$ for $i>0$. 
But $Q_2w_7=c_7\not =0$.
This contradicts to the commutativity of $f_*$ and $Q_2$. 
\end{proof}

\begin{thm}  Let $G$ be a simply connected group such that $H^*(BG)$ has $p$-torsion.
Let $X=X(N)$ be an approximation for $BG$ for 
$N\ge 2p+3$. Then
      $ DH^4(X)\not =0$.
\end{thm}
\begin{proof}
It is only need to prove the theorem when  $G$ is a simple group
having $p$ torsion in $H^*(BG)$.
Let $p=2$.  
It is well known that there is an embedding  $j:G_2\subset G$ such that  (see \cite{Pi-Ya}, \cite{YaR} for details)
\[ H^4(BG)\stackrel{j^*}{\cong} H^4(BG_2)\cong \bZ\{w_4\}.\] 

Let $x=(j^{*})^{-1}w_4\in H^4(BG)$.  
From Lemma 3.1 in \cite{YaR}, we see that $2x$ is represented by Chern classes.
Hence  $2x$ is the image from $CH^*(X)$,  and so 
$2x\in N^1H^4(X)$.  This means there is an open set $U\subset X$ such that
$ 2x=0\in H^*(U)$ that is. $x$ is $2$-torsion in $H^*(U)$. Hence from Lemma 3.5,
we have $x\in N^1H^4(U)$, and so there is $U'\subset U$ such that $x=0\in H^4(U')$.
This implies $x\in N^1H^4(X)$.

Since $j^*(Q_1x)=Q_1w_4=w_7\not=0$, we see $Q_1x\not =0$.
From the main lemma (Lemma 4.1), we see $DH^4(X)\not =0$
for $G$.

For the cases $p=3,5$, we consider the exceptional groups $F_4,E_8$ respectively.  Each simply connected simple group
$G$ contains $F_4$ for $p=3$, $E_8$ for $p=5$. There is $x\in H^4(BG)$ such  that $px$ is a Chern class and $Q_1(x)\not =0\in H^*(BG;\bZ/p).$
In fact,  there is embedding $j:(\bZ/p)^3\subset G$ with $j^*(x)=Q_0(x_1x_2x_3)$. 
Hence we have the theorem.
\end{proof} 
\begin{cor}  Let $X$ be an approximation for $BSpin_n$
with $n\ge 7$ or $BG$ for an exceptional group $G$.  Then $X$ is not stable rational.
\end{cor}
{\bf Remark.}
Kordonskii \cite{Ko}, Merkurjev (Corollary 5.8 in \cite{Me}),
and Reichstein-Scavia show \cite{Re-Sc}
that  $BSpin_n$ itself is stably rational when $n\le 14$.
 These facts do not mean contradictions.
The (Ekedahl) approximation $X$  is not stable rationally equivalent to 
$GL_M/G=BG$ where $M$ is a large number such that $G$ acts freely on $GL_M$.  On the other hand $X$ is  constructed  from 
a quasi projective variety $BG$ as taking 
 intersections of subspaces of $\bP^{M'}$ for a large $M'$.
(The author thanks Federico Scavia who pointed out this remark.)

At last of this section, we consider the case $G=PGL_p$.
We have (for example Theorem 1.5,1.7 in \cite{Ka-Ya})  
additively \[H^*(BG;\bZ/p)\cong 
M\oplus N \quad with \ \ 
 M\stackrel{add.}{\cong}\bZ/p[x_4,x_6,...,x_{2p}],\quad \]
\[ N= SD\otimes \Lambda(Q_0,Q_1)\{u_2\}\quad with\ \ 
SD=\bZ/p[x_{2p+2},x_{2p^2-2p}]\]
where $x_{2p+2}=Q_1Q_0u_2$ and suffix means its degree.  The Chow ring is given as 
\[ CH^*(BG)/p\cong M\oplus SD\{Q_0Q_1(u_2)\}.\]
From Lemma 4.1, we have ;
\begin{thm}  Let $p$ be odd.
For an approximation $X$ for $BPGL_p$,
we see $ \bZ/p\{Q_0u_2\}\subset DH^*(X)$, and moreover there is a surjection \[  \bZ/p[x_{2p^2-2p}]\{Q_0u_2\}\twoheadrightarrow
DH^*(X)/(Im(cl))\quad for\ *<N\]
for the cycle map $cl:CH^*(X)\to H^{2*}(X))$.
\end{thm}



\section{$\bZ/p$-coefficient cohomology for abelian groups }

In the preceding sections, we have seen that cases
$DH^*(X;A)\not =0$ are not so rare 
for $A=\bZ_{(p)}$, $\bZ/p^i$, $i\ge 2$. However currently it seems  difficult to make  
such example for $A=\bZ/p$.   (Recall the remark in the last in  $\S 4$.)
\begin{que}
$DH^*(X;\bZ/p)=0$ for each smooth projective variety $X$ ?
\end{que}

At first, we consider the case $G=(\bZ/p)^3$.
We can take a $quasi$ projective approximation $\bar X(N)$ of $B\bZ/p$ explicitly 
by the quotient (the $N$-dimensional lens space)
\[\bar X(N)=\bC^{N*}/(\bZ/p)\quad where \ \bC^{N*}=(\bC^N-\{0\}).
\]
Next we consider the projective  approximation 
\[X(N)\to \bar X(N)\times \bP^{N}\to B\bZ/p\times \bP^{\infty}.\]
Let us write $X_i$ (resp. $X_i'$) the above $\bar X(N)$ (resp. $\bar X(N-1))$ for a sufficient large number $N$.
Let \[ i_1:Y_1=X_1'\times X_2\times X_3\to
 X=X_1\times X_2\times X_3
.\] 
Similarly we define $Y_2,Y_3$, and the disjoin union $Y=Y_1\sqcup Y_2\sqcup Y_3$.

Recall that for $p:odd$
\[H^*(X;\bZ/p)\cong \bZ/p[y_1,y_2,y_3]/(y_1^{N+1},y_1^{N+1},y_3^{N+1})\otimes
\Lambda(x_1,x_2,x_3),\]
and $H^*(Y_i;\bZ/p)\cong H^*(X;\bZ/p)/(y_i^N)$ for $i=1,2,3$.
For $p=2$, some graded ring $grH^*(X;\bZ/2)$ is isomorphic to the above ring
(in fact $x^2_i=y_i$).

For the embedding $f_i:X_i'\to X_i$, it is known 
$f_{i*}(1)=c_1(N_i)$ where $N_i$ is the normal   bundle for $X_i'\subset X_i$.
Hence the transfer is given by
\[ f_{1*}(1)=y_1,\quad f_{2*}(1)=y_2,\quad f_{3*}(1)=y_3.\]
 Therefore we have for $x=(x_2x_3+ x_3x_1+x_1x_2)\in H^*(Y_1
\sqcup Y_2\sqcup Y_3;\bZ/p)$,
\[ f_*(x)=y_1x_2x_3+y_2x_3x_1+y_3x_1x_2=Q_0(x_1x_2x_3)=\alpha.\]
(Note that the element $x=(x_1x_2+x_2x_3+x_3x_1)$ is not 
in the integral cohomology $H^*(Y)$.)  
Thus we see $\alpha\in \tilde N^cH^*(X;\bZ/p).$
More generally, we see 
\begin{thm}
We have the  approximation $X=X(N)$ for 
$(B\bZ/p)^n$ such that  
\[DH^*(X;\bZ/p)=0 \quad *<N.\]
\end{thm}

We recall here  the motivic cohomology.
By Voevodsky [Vo 1], $H^{*,*'}(B\bZ/p;\bZ/p)$ satisfies the K$\ddot{u}$nneth formula so that (for $p$ odd) 
\[H^{*,*'}(B(\bZ/p)^n;\bZ/p)\cong \bZ/p[\tau, y_1,...,,y_n]/(y_1^{N+1},...,y_n^{N+1})\otimes
\Lambda(x_1,...,x_n).\]
Here $0\not =\tau\in H^{0,1}(Spec(\bC);\bZ/p)$, and
$deg(y_i)=(2,1)$, $deg(x_i)=(1,1)$.

From Theorem 3.1, we can identify
$N^cH_{et}^*(X;\bZ/p)= F_{\tau}^{*,*-c}$
where \\ 
$F_{\tau}^{*,*-c}=Im(\times \tau^c:H^{*,*-c}(X;\bZ/p)\to H^{*,*}(X;\bZ/p)).$

\begin{cor} (Theorem 5.1 in \cite{Te-Ya})  Let  $X=X(N)$ be  the approximation for  $(B\bZ/p)^n$ for a sufficient large $N$.  Then we have
\[ H^*(X:\bZ/p)/N^1H^*(X;\bZ/p)\cong \Lambda(x_1,....,x_n).\]
\end{cor}
\begin{proof}
Let $x\in Ideal(y_1,...,y_n)\subset H^{*,*'}(X;\bZ/p)$.  Then $deg(x)=(*,*')$ with
$*>*'$, and $x$ is a multiplying  of $\tau$. Hence $x\in N^1H^*(X;\bZ/p)$.
\end{proof}
\begin{proof}[Proof of Theorem 7.2.]
Let $x\in N^1H^*(X;\bZ/p)$.  From the above corollary,
$x\in Ideal(y_1,...,y_n)$ which is in the image of the transfer
(such as $Y_1=X_{1}'\times X_2\times ...\times X_n$).
That is $x\in \tilde N^1H^*(X;\bZ/p)$. 
\end{proof}

We can extend Lemma 7.3, by using the following lemma.
Let us write by $XG$ an approximation for $BG$.
\begin{lemma}
Let $G$ have a Sylow $p$-subgroup $S$.
If $ DH^*(XS;\bZ/p)=0$, then 
so for $BG$.
\end{lemma}
\begin{proof}
Let $j: BS\to BG$ and $i; Y\to XS$.  We consider maps
\[ H^*(Y;\bZ/p)\stackrel{i_*}{\to}
  H^*(XS;\bZ/p)\stackrel{j_*}{\to}
H^*(XG;\bZ/p).\]
Here $j_*=cor_S^G$ is the transfer for finite groups.

Note $j^*N^1H^*(XG;\bZ/p)\subset N^1H^*(XS;\bZ/p)$ by the naturality
of $j^*$.
Hence given $x\in N^1H^*(XG;\bZ/p)$. the element 
$y=j^*(x)$ is in $N^1H^*(XS;\bZ/p)$.

 By the assumption
in this lemma, there is $y'\in H^*(Y;\bZ/p)$ with $i_*(y')=y$.
Thus  we have
$\ j_*i_*(y')=j_*y=j_*j^*(x)=[G;S]x.$ 
\end{proof}

Similarly, we can prove ;
\begin{cor}
Let $G$ have  an abelian $p$ Sylow subgroup, and $X=X(N)$
be an approximation for $BG$.  Then
$ DH^*(X;\bZ/p)=0$ for $  *<N$.\
\end{cor}

\section{the groups $Q_8$ and $D_8$}

When $|G|=p^3$, we have the short exact sequence
\[ 0\to C\to G\to V \to 0\]
where $C\cong \bZ/p$ is in the center and $ V\cong \bZ/p\times \bZ/p$. 
Let us take generators such that $C=\la c\ra, V=\la a,b \ra.$
Moreover  we can take $[a,b]=c$ when $G$ is non-abelian.

There are two cases, when $p=2$, the quaternion group $Q_8$
and the dihedral group $D_8$. 
We will show here
\begin{thm} Let $X=X(N)$ be an approximation for $Q_8$
or $D_8$.  Then $DH^*(X;\bZ/2)=0$ for $*<N$.
\end{thm}

{\bf 8.1.}  The case $G=Q_8$. Then  $a^2=b^2=c$.
 Its cohomologies are well known (see \cite{Qu2})
\[H^*(BG)/2\cong \bZ/2[y_1,y_2,c_2]/(y_i^2,y_1y_2)\ \ 
|y_i|=2,\]
\[H^*(BG;\bZ/2)\cong \bZ/2[x_1,x_2,c_2]/(x_1x_2+y_1+y_2,
x_1y_2+x_2y_1)\]
\[\cong \bZ/2\{1,x_1,y_1,x_2,y_2,w\}\otimes \bZ/2[c_2]\]
where $ x_i^2=y_i$ $|x_i|=1$, and  $w=y_1x_2=y_2x_1$,
$|w|=3$.

Therefore, we see 
\[H^*(BG;\bZ/2)/(y_1,y_2,c_2)\cong \bZ/2\{1,x_1,x_2\}.\]
Of course $deg(x_i)=(1,1)$ in $H^{*,*'}(BG;\bZ/2)$ and they are
not in $N^1H^*(BG;\bZ/2)$.  Thus we have Theorem 8.1 for 
$G=Q_8$.

{\bf 8.2.}  The case $G=D_8$.  Then  $a^2=c, b^2=1$.
 It is well known
\[ H^*(BG)/2\cong \bZ/2[y_1,y_2,c_2]/(y_1y_2)\{1,e\}\quad with \ |e|=3.\]
  The mod $2$ cohomlogy is written  \cite{Qu2}
\[ H^*(BG;\bZ/2)\cong \bZ/2[x_1,x_2,u]/(x_1x_2)\quad (with\ |u|=2)\]
\[ \cong (\oplus_{j=1}^2 \bZ/2[y_j]\{y_j, x_j,y_ju,x_ju\}\oplus
\bZ/2\{1,u\})\otimes \bZ/2[c_2].\]
Here $y_j=x_j^2, u^2=c_2$ and $Q_0(u)=(x_1+x_2)u=e,
Q_1Q_0(u)=(y_1+y_2)c_2$.

We note $y_1,y_2,c_2\in CH^*(BG)/2$ and
\[ H^*(BG;\bZ/2)/(y_1,y_2,c_2)\cong
(\oplus _{j=1}^2\bZ/2\{x_j,x_ju\})\oplus \bZ/2\{1,u\}.\]
Moreover, $deg(x_j)=(1,1)$, $deg(u)=(2,2)$ in the motivic cohomology \\ 
$H^{*,*'}(BG;\bZ/2)$ and they are not in $N^1
H^*(BG;\bZ/2)$.
Here we note $deg(x_ju)=(3,3)$, but there is $u_j'\in 
H^{3,2}(BG;\bZ/2)$ with $x_ju=\tau u_j'$ from Lemma 6.2 in \cite{YaJ}
( i.e., $x_ju\in N^1H^*(X;\bZ/2)$).

Hence for the proof of Lemma 8.1 (for $G=D_8$),
 it is only need to show 
\begin{lemma} We have $x_iu\in \tilde N^1H^*(BG;\bZ/2)$.
\end{lemma}

To prove the above lemma, for a $G-$variety
$H$, we consider the equivariant cohomology 
\[ H_G^*(H;\bZ/p)=H^*(E(N)\times _{G}H;\bZ/p)\]
where $E(N)$ is an (approximation of) contractible free $G$-variety.  Let us write
\[ X_GH=appro. \ of\ E(N)\times_GH \ so\ that\ H_G^*(H;\bZ/p)\cong H^*(X_GH;\bZ/p).\]
For a closed embedding  $i:H\subset K$ of $G$-varieties, we can
define the transfer
\[ i_*: H_G^*(H;\bZ/p)\to H^*_G(K;\bZ/p)
\quad by \  i:X_GH\stackrel{id\times_Gi}{\to} X_GK. \]

Hereafter in this section, let $G=D_8$.
We recall arguments in \cite{YaJ}.
We define the $2$-dimensional representation 
$\tilde c: G\to U_2$ such that 
$\tilde c(a)=diag(i,-i)$
and $\bar c(b)$ is the permutation matrix $(1,2)$. 
By this representation, we identify that
$W=\bC^{2*}=\bC^2-\{0\}$
 is an $G$-variety. Note $G$ acts freely on $W\times \bC^*$
but it does not so on $W=\bC^{2*}.$

 The fixed points set on $W$ under $b$ is
\[  W^{\la b\ra}=\{(x,x)|x\in \bC^*\}=\bC^*\{e'\} ,\quad e'=diag(1,1)\in GL_2(\bC).\]
Similarly $W^{\la bc\ra}=\bC^*\{a^{-1}e'\}.$  
  Take
\[ H_0=\bC^{*}\{e',ae'\},\quad H_1=\bC^{*}\{g^{-1}e',q^{-1}ae'\}\]
where $g\in GL_2(\bC)$ with $g^{-1}bg=ab$ (note $(ab)^2=1$).

Let us write
$ H=H_0\sqcup H_1.$
Then $G$ acts on $H_i$ and acts freely on $\bC^{2*}-H$.
In fact it does not contain
fixed points of non-trivial stabilizer groups.

We consider the transfer for some $G$-set $H$ in 
$\bC^{2*}$, and induced equivariant cohomology
\[ i_*: H^*_G(H;\bZ/2)\to H^*_G(\bC^{2*};\bZ/2).\]
\begin{lemma} We have 
\[H^*_G(H_0;\bZ/2)\cong \bZ/2[y] \otimes \Lambda(x,z)\quad
 with \ y=x^2,\ |x|= |z|=1.\]
\end{lemma}
\begin{proof}
We consider the group extension $0\to \la a
\ra \to G\to \la b \ra \to 0$ and the induced spectral sequence
\[E_2^{*,*'}=H^*(B\la b\ra; H^{*'}_{\la a\ra }(H_0;\bZ/2))\Longrightarrow
H_G^*(H_0;\bZ/2).\]

Since $\la a\ra\cong \bZ/4$ acts freely on $H_0$, we see
\[H_0/\la a\ra \cong \bC^{*}\{e',ae'\}/\la a\ra
 \cong \bC^{*}/\la a\ra \cong \bC^{*}.\]
Therefore we have 
\[H^*_{\la a\ra}(H_0;\bZ/2)\cong H^*(\bC^*/\la a\ra;\bZ/2)
\cong H^*(\bC^*;\bZ/2)\cong \Lambda (z)\quad |z|=1.\]
Since $\la b\ra$ acts trivially on $\Lambda(z)$ we have
$ H^*_G(H;\bZ/2)\cong H^*(B\la b\ra;\bZ/2)\otimes \Lambda(z).$
\end{proof}

Note $H_G^*(H_0;\bZ/2)\cong H_G^*(H_1;\bZ/2)$ and hence we 
see
\[H^*_G(H;\bZ/2)\cong  \oplus_{j=1}^2 \bZ/2[y_j]\{1_j,y_j, x_j,x_jz_j,z_j\}.\]

We consider the long exact sequence
\[ (*)\quad ...\to H^{*}_G(\{0\};\bZ/2)\stackrel{i_*=c_2}{\to}
H^{*+4}_G(\bC^2;\bZ/2) \to H^{*+4}_G(\bC^{2*};\bZ/2)
\to ...\]
and we have  $H^*_G(\bC^{2*};\bZ/2)\cong H^*(BG;\bZ/2)/(c_2)$. 
Hence,   we get 
\[H^*_G(\bC^{2*};\bZ/2)\cong ( \oplus_{j=1}^2 \bZ/2[y_j]\{y_j,x_j,x_ju_j',u_j'\})\oplus
\bZ/2\{1,u\}\]

Now we consider the transfer 
$H^*_G(H;\bZ/2) \stackrel{i_*}{\to} H^{*+2}_G(\bC^{2*};\bZ/2)$. 
We have explicitly (page 527 in \cite{YaJ})
\[ i_*(1_j)=y_j,\ i_*(x_j)=y_jx_j,\ i_*(x_jz_j)=x_ju_j',\ i_*(z_j)=u_j'.\]  
Therefore we have Theorem 8.1 for $G=D_8$.

To see the above fact we recall 
 the long exact
sequence for $i:H \subset \bC^{2*}$
\[ (**)\quad ...\to H^{*+1}_G(\bC^{2*}-H;\bZ/2)\stackrel{\delta}{\to}
H^*_G(H;\bZ/2) \stackrel{i_*}{\to} H^{*+2}_G(\bC^{2*};\bZ/2)\]
\[
\stackrel{j^*}{\to} H^{*+2}_G(\bC^{2*}-H;\bZ/2)\to...\]
The transfer $i_*$ is determined by the following lemma.
\begin{lemma}  In the above $(**)$, we see 
$\delta=0$, and hence $i_*$ is injective.
\end{lemma}
\begin{proof}
Since $G$ acts freely on $\bC^{2*}-H$, we have
\[ H^*_G(\bC^{2*}-H:\bZ/2)\cong H^*((\bC^{2*}-H)/G;\bZ/2),\]
which is zero when $*>4=2dim((\bC^{2*}-H)/G)$.   Hence $\delta$ must be zero 
for $*>4$, and $i_*$ is injective for $*>4$.  In particular, $i_*(y^2_jz_j)=y_j^2u_j'$.  Since  $H^*_G(H;\bZ/2)$ is 
$\bZ/2[y_1]$-free (or $Z/2[y_2]$-free,) we see
$i_*(z_j)=u_j'$.
\end{proof}

\end{document}